\documentclass[12pt]{amsart}
\usepackage{amssymb}
\def\Ind{\mathop{\rm Ind}\nolimits}
\def\thetabar{\bar{\theta}}
\def\x{\chi}

\newbox\Gbox
\setbox\Gbox=\hbox{$G$}
\newdimen\Gwidth
\Gwidth=\wd\Gbox
\multiply\Gwidth by 85
\divide\Gwidth by 100
\newdimen\Gheight
\Gheight=\ht\Gbox
\newcommand\Gbar{G\kern-\Gwidth\overline{\phantom{\vrule height\Gheight width\Gwidth}}}

\newbox\Qbox
\setbox\Qbox=\hbox{$Q$}
\newdimen\Qwidth
\Qwidth=\wd\Qbox
\multiply\Qwidth by 85
\divide\Qwidth by 100
\newdimen\Qheight
\Qheight=\ht\Qbox
\newcommand\Qbar{Q\kern-\Qwidth\overline{\phantom{\vrule height\Qheight width\Qwidth}}}

\newbox\Bbox
\setbox\Bbox=\hbox{$B$}
\newdimen\Bwidth
\Bwidth=\wd\Bbox
\multiply\Bwidth by 80
\divide\Bwidth by 100
\newdimen\Bheight
\Bheight=\ht\Bbox
\newcommand\Bbar{B\kern-\Bwidth\overline{\phantom{\vrule height\Bheight width\Bwidth}}}

\newbox\Abox
\setbox\Abox=\hbox{$A$}
\newdimen\Awidth
\Awidth=\wd\Abox
\multiply\Awidth by 80
\divide\Awidth by 100
\newdimen\Aheight
\Aheight=\ht\Abox
\newcommand\Abar{A\kern-\Awidth\overline{\phantom{\vrule height\Aheight width\Awidth}}}

\newcommand\bbar{\overline{b}}
\newcommand\Tbar{\overline{T}}
\newcommand\tbar{\overline{t}}
\newcommand\Thetabar{\overline{\Theta}}

\newcommand\iso{\isomorphism}
\newcommand\F{\mathbb{F}}

\newcommand\Z{\mathbb{Z}}
\newcommand\R{\mathbb{R}}
\renewcommand\H{\mathbb{H}}
\newcommand\C{\mathbb{C}}

\let\isomorphism\cong
\renewcommand\cong{\equiv}
\newcommand\sset{\subseteq}
\newcommand\SO{{\rm SO}}

\renewcommand\a{\alpha}
\def\gend#1{\langle#1\rangle}

\numberwithin{equation}{section}
\newtheorem{theorem}{Theorem}[section]
\newtheorem{lemma}[theorem]{Lemma}
\newtheorem{corollary}[theorem]{Corollary}
\theoremstyle{remark}
\newtheorem{example}[theorem]{Example}
\newtheorem{remark}[theorem]{Remark}
\DeclareMathOperator{\Aut}{{\rm Aut}}
\DeclareMathOperator{\Out}{{\rm Out}}
\DeclareMathOperator{\End}{{\rm End}}
\newcommand{\GL}{{\rm GL}}
\newcommand{\PGL}{{\rm PGL}}

\newcommand\SL{\mathop{\rm SL}\nolimits}

\def\mathllap#1{\mathchoice
{\llap{$\displaystyle #1$}}%
{\llap{$\textstyle #1$}}%
{\llap{$\scriptstyle #1$}}%
{\llap{$\scriptscriptstyle #1$}}}

\def\presentation#1#2{\bigl\langle#1\mathllap{\phantom{#2}}\bigm|#2\mathllap{\phantom{#1}}\bigr\rangle}

\begin{document}

\title{Spherical space forms revisited}
\thanks{Supported by NSF grant DMS-1101566}
\author{Daniel Allcock}
\address{Department of Mathematics,\\University of Texas, Austin}
\email{allcock@math.utexas.edu}
\urladdr{http://www.math.utexas.edu/\textasciitilde allcock}
\date{9 September 2016}

\begin{abstract}
We give a simplified proof of J.~A.~Wolf's classification of finite
groups that can act freely and isometrically on a round sphere of some
dimension.  We slightly improve the classification by removing
some non-obvious redundancy.  The groups are the same as the Frobenius complements of
finite group theory.
\end{abstract}

\maketitle

\noindent
In chapters 4--7 of his famous {\it Spaces of Constant Curvature}
\cite{Wolf}, J. A. Wolf classified the spherical space forms: connected Riemannian
manifolds locally isometric to the $n$-sphere $S^n$.
By passing to the action of the fundamental group on the universal
covering space, this is equivalent to classifying the possible free
isometric actions of finite groups on $S^n$.  Then, by embedding $S^n$
in Euclidean space, this is equivalent to classifying the real
representations of finite groups that are ``free'' in the sense
that no element except the identity fixes any vector except~$0$.  This
allowed Wolf to use the theory of finite groups and their
representations.

Our first goal is to give a simplified proof of Wolf's classification
of the finite groups $G$ that can act freely and isometrically on
spheres.  Wolf's main result here was the list of presentations in
theorems~6.1.11 and~6.3.1 of \cite{Wolf}.  Our approach to Wolf's
theorem leads to the most interesting example, the binary icosahedral
group, with very little case analysis and no character theory.  (The
trick consists of the equalities \eqref{eq-the-trick-first-version}
and \eqref{eq-the-trick-second-version} in the proof of lemma~\ref{lem-trick-for-large-Sylow-2-subgroups}.
These rely on an elementary
property of the binary tetrahedral group, stated in
lemma~\ref{lem-free-actions-of-binary-tetrahedral-group}.)

Our second goal is to remove the redundancy from Wolf's list; this
modest improvement appears to be new.  Some groups appear
repeatedly on Wolf's list because different presentations can define
isomorphic groups.  See example~\ref{example-Wolf-duplication} for
some non-obvious isomorphisms.
It would not be  hard to  just
work out the isomorphisms among the groups defined by Wolf's
presentations. But  it is more natural to
reformulate the classification in terms of intrinsically defined
subgroups.  Namely, a finite group $G$ acts freely and isometrically
on a sphere if and only if it has one of 6 possible ``structures'', in
which case it has a unique such ``structure'' up to conjugation.  See
theorems \ref{thm-the-groups} and~\ref{thm-uniqueness-of-structure}.
In fact we parameterize the possible $G$ without redundancy,
in terms of a fairly simple set of invariants.  Namely: a type
\ref{type-I}--\ref{type-VI}, two numbers $|G|$ and $a$, and a subgroup
of the unit group of the ring $\Z/a$.  In a special case one must also
specify a second such subgroup.  See
section~\ref{sec-constructive-classification} for details.

The full classification of spherical space forms requires not just the
list of possible groups, but also their irreducible free actions on
real vector spaces, and how their outer automorphism groups permute these representations. See \cite[Thm~5.1.2]{Wolf} for why
this is the right data to tabulate and \cite[Ch.~7]{Wolf} for the
actual data for each group.  We expect that this data could
be described cleanly in terms of our descriptions of the groups,
but have not worked out the details.  It would remain
lengthy, because of many cases and subcases to consider.

For many authors the phrase ``spherical space form'' means a quotient
of a sphere by a free action of a finite group of homeomorphisms or
diffeomorphisms, rather than isometries.  In this paper we consider
only isometric actions on round spheres.  See \cite{Wall} for the rich topology and group theory involved in
the more general theory.

Expecting topologists and geometers rather than group theorists as
readers, we have made the paper self-contained, with three exceptions.
First, we omit proofs of Burnside's transfer theorem and the
Schur-Zassenhaus theorem. Second, we use the fact that $\SL_2(\F_5)$
is the unique perfect central extension of the alternating group $A_5$
by $\Z/2$, giving a citation when needed.  Third, we use
$\Aut\SL_2(\F_5)=\PGL_2(\F_5)\iso S_5$, which is just an exercise.

Finite group theorists study the same groups Wolf did, from a
different perspective.  We will sketch the connection briefly because
our descriptions of the groups may have some value in this context.  A
finite group $G$ is called a Frobenius complement if it acts
``freely'' on some finite group $H$, meaning that no element of~$G$
except~$1$ fixes any element of~$H$ except~$1$.  To our knowledge, the
structure of Frobenius complements (in terms of presentations) is due
to Zassenhaus \cite{Zassenhaus-old}. Unfortunately his paper contains
an error, and the first correct proof is due to Passman \cite[Theorems
  18.2 and~18.6]{Passman}.  See also Zassenhaus' later paper
\cite{Zassenhaus-new}.  If $G$ is a Frobenius complement, then after a
preliminary reduction one can show that $H$ may be taken to be a
vector space over a finite field $\F_p$, where $p$ is a prime not
dividing~$|G|$. Our arguments apply with few or no changes; see
remark~\ref{remark-Frobenius-complements} and also \cite{Meierfrankenfeld}, especially Prop.\ 2.1.

\medskip
We will continue to abuse language by speaking of free actions on
vector spaces when really we mean that the action is free away
from~$0$.   We will use the standard notation $G'$ for the
commutator subgroup of a group $G$, and $O(G)$ for the unique
maximal-under-inclusion odd normal subgroup when $G$ is finite.

We will also use ATLAS notation for group structures \cite{ATLAS}.
That is, if a group $G$ has a normal subgroup $A$, the quotient by
which is $B$, then we may say ``$G$ has structure $A.B$''.  We
sometimes write
this as $G\sim A.B$. This usually does not completely describe $G$, because
several nonisomorphic groups may have ``structure $A.B$''.
Nevertheless it is a helpful shorthand, especially if $A$ is
characteristic.  If the group extension splits then we may write $A:B$
instead, and if it doesn't then we may write $A\cdot B$.  See
theorem~\ref{thm-the-groups} for a some examples.  When we write $A:B$, we will
regard $B$ as a subgroup of $G$ rather than just a quotient.  (In all
our uses of this notation, the complements to $A$ turn out to be
conjugate, so there is no real ambiguity in choosing one of them.)

I am very grateful to the referee for catching a serious error in an
earlier version of this paper.

\section{The groups}
\label{sec-groups}

\noindent
It is well-known that the groups of orientation-preserving isometries
of the tetrahedron, octahedron (or cube) and icosahedron (or
dodecahedron) are subgroups of $\SO(3)$ isomorphic to $A_4$, $S_4$ and
$A_5$.  The preimages of these groups in the double cover of $\SO(3)$
are called the binary tetrahedral, binary octahedral and binary
icosahedral groups.  They have structures $2.A_4$, $2.S_4$ and
$2.A_5$, where we are using another ATLAS convention: indicating a
cyclic group of order $n$ by simply writing $n$; here $n$ is~$2$.
These are the only groups with these structures that we will encounter
in this paper.  So we abbreviate them (again following the ATLAS) to
$2A_4$, $2S_4$ and $2A_5$, and specify that this notation refers to
the binary polyhedral groups, rather than some other groups with
structure $2.A_4$, $2.S_4$ or $2.A_5$.  Alternate
descriptions of the binary tetrahedral and binary icosahedral groups
are $2A_4\iso\SL_2(3)$ and $2A_5\iso\SL_2(5)$.

It is also well-known that the double cover of $\SO(3)$ may be
identified with the unit sphere $\H^*$ in Hamilton's quaternions $\H$.
A finite subgroup of $\H^*$ obviously acts by left multiplication on
$\H^*$.  So $2A_4$, $2S_4$ and $2A_5$ act freely on the unit sphere~$S^3$.

Similarly, $\SO(3)$ contains dihedral subgroups, and their preimages
in $\H^*$ are called binary dihedral.  If we start with the dihedral
group of order $2n$, then the corresponding binary dihedral group of order~$4n$ can
be presented by
$$
\presentation{x,y}{x^{2n}=1,\ y x y^{-1}=x^{-1},\ y^2=x^n}.
$$ 
This group may be identified
with a subgroup of $\H^*$ by taking 
\begin{equation}
\label{eq-binary-dihedral-representation}
x\mapsto\hbox{(any primitive $2n$th root
of unity in $\R\oplus\R i$)}\qquad y\mapsto j.
\end{equation}
Left multiplication by these elements of $\H^*$ gives a free action on~$S^3$.
(Replacing the root of unity by its
inverse gives an equivalent representation.)
Restricting to the case $n=2^{m-2}$, $m\geq3$, one obtains the
quaternion group $Q_{2^m}$ of order $2^m$.  Some authors call this a
generalized quaternion group, with ``quaternion group'' reserved for
$Q_8$.
  
Now we can state our version of Wolf's theorems 6.1.11 and~6.3.1,
supplemented by a uniqueness theorem. An $n$-element of a
group means an element of order~$n$.  

\begin{theorem}[Groups that act freely and isometrically on spheres]
\label{thm-the-groups}
Suppose $G$ is a finite group that acts freely and isometrically on a
sphere of some dimension.  Then it has one of the following six
structures, where $A$ and $B$ are cyclic groups whose orders are odd
and coprime, every nontrivial Sylow subgroup of $B$ acts nontrivially on $A$, and
every prime-order element of $B$   acts trivially on $A$.
\renewcommand\theenumi{{\rm\Roman{enumi}}}
\begin{enumerate}
\item
\label{type-I}
$A:\bigl(B\times\hbox{\rm (a cyclic $2$-group $T$)}\bigr)$, where
if $T\neq1$ then its involution fixes
$A$ pointwise.
\item
\label{type-II}
$A:\bigl(B\times\hbox{\rm (a quaternionic group $T$)}\bigr)$.
\item
\label{type-III}
$(Q_8\times A):(\Theta\times B)$, where $\Theta$ is a cyclic $3$-group
which acts nontrivially on $Q_8$ and whose $3$-elements centralize
$A$, and $|A|$ and $|B|$ are prime to~$3$.
\item
\label{type-IV}
$\bigl( (Q_8\times A):(\Theta\times B)\bigr)\cdot2$, where $\Theta$,
$|A|$ and $|B|$ are as in \eqref{type-III}, and the
quotient $\Z/2$ is the image of a subgroup $\Phi$ of $G$, isomorphic to
$\Z/4$, whose $4$-elements act by an outer automorphism on $Q_8$, by
inversion on $\Theta$ and trivially on~$B$.
\item
\label{type-V}
$2A_5\times(A:B)$ where $|A|$ and $|B|$ are prime to~$15$.
\item
\label{type-VI}
$\bigl(2A_5\times(A:B)\bigr)\cdot2$, where $|A|$ and $|B|$ are prime to~$15$, 
and the
quotient $\Z/2$ is the image of a subgroup $\Phi$ of $G$, isomorphic to
$\Z/4$, whose $4$-elements act by an outer automorphism on $2A_5$ and
trivially on $B$. 
\end{enumerate}
Conversely, any group with one of these structures acts freely and
isometrically on
a sphere of some dimension.
\end{theorem}

These groups are parameterized in terms of simple invariants in
theorem~\ref{thm-irredundant-list}.  A binary dihedral group has type
\ref{type-I} or~\ref{type-II}, $2A_4$ has type~\ref{type-III}, $2S_4$
has type~\ref{type-IV}, and $2A_5$ has type~\ref{type-V}.

\begin{theorem}[Uniqueness of structure]
\label{thm-uniqueness-of-structure}
The structure in theorem~\ref{thm-the-groups} is unique, in the following sense:
\begin{enumerate}
\item
\label{item-cases-disjoint}
Groups of different types \ref{type-I}--\ref{type-VI} cannot be isomorphic.
\item
\label{item-uniqueness-of-structure}
Suppose $G$ has  one of the types
\ref{type-I}--\ref{type-VI}, with respect to
some subgroups $A$, $B$ (and whichever of $T$, $Q_8$, $\Theta$, $\Phi$ and $2A_5$ are
relevant), and also with respect to some subgroups $A^*$, $B^*$ (and
$T^*$, $Q_8^*$, $\Theta^*$, $\Phi^*$ and $2A_5^*$, when relevant).  Then 
some element of $G$ conjugates every unstarred group to
the corresponding starred group.
In particular, $A^*=A$ (and $Q_8^*=Q_8$ and $2A_5^*=2A_5$, when
relevant).
\end{enumerate}
\end{theorem}

\begin{remark}[Correspondence with Wolf's types]
\label{remark-correspondence-with-Wolf}
Our types correspond in the obvious way to Wolf's in
theorems 6.1.11 and~6.3.1 of \cite{Wolf}.  However, his generator $A$
might not generate our subgroup~$A$ and his generator $B$ might not
even lie in our
subgroup $B$.
\end{remark}

\begin{remark}[Implied relations]
\label{remark-implied-relations}
Some useful information is implicit.  For example, $B$
acts trivially on $Q_8$ for types \ref{type-III} and
\ref{type-IV}, because $Q_8$ has no automorphisms
of odd order${}>3$.  Also, $\Theta$ must act trivially on $A$ for type~\ref{type-IV}, because $A$ has abelian automorphism
group and $\Theta\leq G'$.  In light of
these remarks, one could rewrite $G$'s structure for type~\ref{type-IV} as
$\bigl((Q_8:\Theta)\times(A:B)\bigr)\cdot2$.  This would be more
informative, but hide the relationship to type~\ref{type-III}.

In all cases, $G$ has at most one involution.  This is obvious
except for type~\ref{type-II}.  Then,
$T$'s central involution lies in $T'$, which centralizes~$A$ because $\Aut
A$ is abelian.  
\end{remark}

At this point we will prove the easy parts of the theorems, namely
that the listed groups do act freely on spheres, and that groups of
different types cannot be isomorphic.  The proof that every group
acting freely on a sphere has structure as in theorem~\ref{thm-the-groups} appears in
sections~\ref{sec-perfect}--\ref{sec-imperfect}, and theorem~\ref{thm-uniqueness-of-structure}'s
uniqueness statement is a
byproduct of the proof.

\begin{proof}[Proof of ``Conversely$\ldots$'' in theorem~\ref{thm-the-groups}]
Define $H$ as the normal subgroup of $G$ that is generated by the
elements of prime order.  We claim: if $H$ has a free action on a real
vector space $V$, then $G$ acts freely on its representation $W$
induced from $V$.  To see this, recall that as a vector space, $W$ is
a direct sum of copies of $V$, indexed by $G/H$.  And $H$'s actions on
these copies of $V$ are isomorphic to the representations got by
precomposing $H\to\GL(V)$ by automorphisms $H\to H$ arising from
conjugation in $G$.  In particular, $H$ acts freely on $W$.  We claim
that $G$ also acts freely on $W$.  Otherwise, some prime-order element
of $G$ would have a fixed point.  But it would also  lie in $H$, which
acts freely.

Now it suffices to determine $H$ and show that it always has a free
action.  For types~\ref{type-I}--\ref{type-II}, $H$ is a cyclic group,
and for types~\ref{type-V}--\ref{type-VI} it is
$2A_5\times\hbox{(cyclic group of order prime to 30)}$.  For types
\ref{type-III}--\ref{type-IV}, $H$ is either cyclic or
$2A_4\times\hbox{(cyclic group of order prime to 6)}$, according to
whether $|\Theta|>3$ or $|\Theta|=3$.  These claims are all easy,
using the facts that $G$'s involution is central (if one exists) and
the prime-order elements of $B$ and (if relevant) $\Theta$ act
trivially on $A$.

Cyclic groups obviously admit free actions.  For a product of $2A_4$
or $2A_5$ by a cyclic group of coprime order, we identify each factor with a
subgroup of $\H^*$ and make $2A_4$ or $2A_5$ act on $\H$ by left multiplication
and the cyclic group act by right multiplication.  If there were a
nonidentity element of this group with a nonzero fixed vector, then
there would be one having prime order, hence lying in one of the
factors.  But this is impossible since each factor acts freely.
\end{proof}

\begin{proof}[Proof of theorem~\ref{thm-uniqueness-of-structure}(\ref{item-cases-disjoint})]
Suppose $G$ has one of the types
\ref{type-I}--\ref{type-VI}.  Then the
subgroup $H$ generated by $A$, $B$ and the index~$3$ subgroup of
$\Theta$ (for types
\ref{type-III}--\ref{type-IV}) is
normal in $G$ and has odd order.  The quotient $G/H$ is a cyclic
$2$-group, a quaternionic group, $2A_4$, $2A_4\cdot2$, $2A_5$ or
$2A_5\cdot2$ respectively.  None of these has an odd normal subgroup
larger than $\{1\}$.  Therefore $H$ is all of $O(G)$, and the isomorphism class of
$G/O(G)$ distinguishes the types.
\end{proof}

\section{Preparation}
\label{sec-preparation}

\noindent
In this section we suppose $G$ is a finite group.  We will establish
general properties of $G$ under hypotheses related to $G$ having a
free action on a sphere.  The results before
lemma~\ref{lem-metacyclic-decomposition} are standard, and included
for completeness.  Lemma~\ref{lem-metacyclic-decomposition} is a
refinement of a standard result.

\begin{lemma}[Unique involution]
\label{lem-unique-involution}
Suppose $G$ has a free action on a sphere.  Then it has at most one
involution.
\end{lemma}

\begin{proof}
Choose a free (hence faithful) action of $G$ on a real vector space~$V$.  An involution has eigenvalues $\pm1$, but $+1$ cannot appear
by freeness.  So there can be only one involution, acting by
negation. 
\end{proof}

\begin{lemma}
\label{lem-abelian-subgroups-are-cyclic}
Suppose $G$ has a free action on a sphere. Then
every abelian subgroup is cyclic, and so is every subgroup of
order $p q$, where $p$ and $q$ are primes.
\end{lemma}

\begin{proof} 
Fix a free action of $G$ on a real vector space $V$, and let $V_\C$ be
its complexification.  $G$ also acts freely on $V_\C$.  Otherwise,
some nontrivial element has a nonzero fixed vector, hence has the real
number $1$ as an
eigenvalue, hence fixes  a nonzero real vector.

Now suppose $A\leq G$ is abelian, and decompose $V_\C$ under $A$ as a
sum of 1-dimensional representations.  By freeness, each of these is
faithful.  So $A$ is a subgroup of the multiplicative group $\C-\{0\}$,
hence cyclic, proving the first claim.  Since any group of
prime-squared order is abelian, this also proves the $p=q$ case of the
second claim.

So suppose $p<q$ are primes and consider a subgroup of $G$ with order
$p q$; by discarding the rest of~$G$ we may suppose without loss that
this subgroup is all of~$G$.
Write $P$ resp.\ $Q$ for a Sylow $p$-subgroup
resp.\ $q$-subgroup.  By Sylow's theorem, $P$ normalizes $Q$.  If it acts
trivially on $Q$ then $G$ is abelian, so suppose $P$  acts
nontrivially on~$Q$.

For purposes of this proof, a character of $Q$ means a homomorphism
$Q\to\C^*$.  It is standard that any complex representation of $Q$ is
the direct sum of $Q$'s character spaces, meaning the subspaces on
which $Q$ acts by its various characters.   Fix a character $\x$ of $Q$
whose character space contains a nonzero vector $v\in V_\C$; $\x$ is faithful since $Q$ acts freely
on $V_\C$.  If $g\in P$ then $g(v)$ lies in the character space for
the character $\x\circ i_g^{-1}:Q\to\C^*$, where $i_g:x\mapsto g x
g^{-1}$ means conjugation by $g$.  Since $\x$ is faithful, and $P$
acts faithfully on $Q$, the various characters $\x\circ i_g^{-1}$ are
all distinct.  Therefore the terms in the sum $\sum_{g\in P}g(v)$ are
linearly independent, so the sum is nonzero.  But this contradicts
freeness since the sum is obviously $P$-invariant.
\end{proof}

\begin{lemma}[Sylow subgroups]
\label{lem-Sylow-subgroups-are-cyclic-or-quaternion}
Suppose all of $G$'s abelian subgroups are
cyclic.  Then its odd Sylow subgroups are cyclic, and its Sylow
$2$-subgroups are cyclic or quaternionic.
\end{lemma}

\begin{proof}
It suffices to treat the case of $G$ a $p$-group, say of order $p^n$.
We proceed by induction on $n$, with the cases $n\leq2$ being
trivial.  So suppose $n>2$.

First we treat the special case that $G$ contains a cyclic group $X$
of index~$p$.  If $G$ acts trivially on $X$ then $G$ is abelian, hence
cyclic.  
So we may assume that $G/X$ is identified with a
subgroup of order~$p$ in $\Aut X$.
Recall that $\Aut X$ is cyclic of order $(p-1)p^{n-2}$ if $p$ is odd,
and $\Z/2$ times a cyclic group of order $2^{n-3}$ if $p=2$.
This shows that  some $y\in G-X$ acts on
$X$ by the $\lambda$th power map, where
$$
\lambda=\begin{cases}
p^{n-2}+1&\hbox{if $p$ is odd}\\
\hbox{$-1$\quad or\quad $2^{n-2}\pm1$}&\hbox{if $p=2$},
\end{cases}
$$ with the possibilities $2^{n-2}\pm1$ considered only if $n>3$.
Write $X_0$ for the subgroup of $X$ centralized by $y$.  Now,
$\gend{X_0,y}$ is abelian, hence cyclic.  The index of its subgroup
$X_0$ is~$p$, because $y^p$ lies in $X$ and centralizes $y$.  Since
$y\notin X_0$, $y$ generates $\gend{X_0,y}$.  We write $p^t$ for the
index of $X_0$ in $X$, which can be worked out from $y$'s action on
$X$.  Namely,
$$
p^t=
\begin{cases}
  2^{n-2}&\hbox{if $p=2$, and $\lambda=-1$ or $2^{n-2}-1$}\\
  p&\hbox{otherwise}
\end{cases}
$$
We choose a generator $x$ for $X$ such that $y^p=x^{-p^t}$.

Since $x y$ and $y$ have the same centralizer in $X$, the same argument
shows that $(x y)^p$ also generates $X_0$.  Now,
\begin{align*}
(x y)^p
&{}=x(y x y^{-1})(y^2 x y^{-2})\cdots(y^{p-1} x y^{1-p})y^p\\
&{}=x\cdot x^\lambda\cdot x^{\lambda^2}\cdots x^{\lambda^{p-1}}\cdot x^{-p^t}.
\end{align*}
Our two descriptions $\gend{y^p}$ and $\gend{(xy)^p}$ of $X_0\leq X\iso\Z/p^{n-1}$ 
must coincide, so
$\mu:=1+\lambda+\cdots+\lambda^{p-1}-p^t$ generates the same subgroup of
$\Z/p^{n-1}$ as $p^t$ does.  One computes
$$
\begin{tabular}{ccl}
  $p^t$&$\mu$\\
  \noalign{\smallskip}
  \cline{1-2}
  \noalign{\smallskip}
$p$&$\quad\binom{p}{2}p^{n-2}$\quad&if $\lambda=p^{n-2}+1$ (including the case
    $2^{n-2}+1$)\\
$2^{n-2}$&$0$&if $\lambda=2^{n-2}-1$\\
$2^{n-2}$&$2^{n-2}$&if $\lambda=-1$
\end{tabular}
$$ Only in the last case do $p^t$ and $\mu$ generate the same subgroup
of $\Z/p^{n-1}$.  So $p=2$, $y$ inverts $X$, and $y^2$ is the
involution in $X$. That is, $G$ is quaternionic.  This
finishes the proof in the special case.

Now we treat the general case.  Take $H$ to be a subgroup of
index~$p$.  If $p$ is odd then $H$ is cyclic by induction, so the
special case shows that $G$ is cyclic too.  So suppose $p=2$.  By
induction, $H$ is cyclic or quaternionic.  If it is cyclic then the
special case shows that $G$ is cyclic or quaternionic,  as desired. So suppose $H$
is quaternionic and take $X$ to be a $G$-invariant index~$2$ cyclic
subgroup of $H$.  This is possible because $H$ contains an odd number
of cyclic subgroups of index~$2$ (three if $H\iso Q_8$ and one
otherwise).  Now we consider the action of $G/X$ on $X$.  If some
element of $G-X$ acts trivially then together with $X$ it generates an
abelian, hence cyclic, group, and the special case applies.  So $G/X$
is $2\times2$ or $4$ and embeds in $\Aut X$.  (These cases require
$|X|\geq8$ or $16$ respectively, or equivalently $n\geq5$ or~$6$.)
Furthermore, $\Aut X$ contains just three involutions, and only one of
them can be a square in $\Aut X$, namely the $(1+2^{n-2})$nd power
map.  Therefore, either possibility for $G/X$ yields an element $y$ of
$G$ which acts on $X$ by this map and has square in $X$.  But then
$\gend{X,y}$ is neither cyclic nor quaternionic, contradicting the special
case.
\end{proof}

Recall that a group is called perfect if its abelianization is trivial.

\begin{lemma}[$2A_5$ recognition]
  \label{lem-G-is-2A5}
  Suppose $G$ is perfect with center $\Z/2$, and every
  noncentral cyclic subgroup has binary dihedral normalizer.  Then
  $G\iso2A_5$.
\end{lemma}

\begin{proof}
 Write $2g$ for $|G|$.  The
  hypothesis on normalizers shows that distinct maximal cyclic
  subgroups of $G$ have intersection equal to $Z(G)$.  So $G$ is
  the disjoint union of $Z(G)$ and the subsets $C-Z(G)$ where $C$ varies over
  the maximal cyclic subgroups of $G$.  We choose representatives
  $C_1,\dots,C_n$ for the conjugacy classes of such
  subgroups and write $2c_1,\dots,2c_n$ for their orders.
  The numbers $c_1,\dots,c_n$ are pairwise coprime because each of
  $C_1,\dots,C_n$ is the centralizer of each of its subgroups 
  of order${}>2$.
  We number the $C_i$ so that $c_1$ is
  divisible by~$2$, $c_2$ is divisible by the smallest prime involved
  in $g$ but not $c_1$, $c_3$ is divisible by the smallest prime
  involved in $g$ but neither $c_1$ nor $c_2$, and so on.
  In particular, $c_i$ is
  at least as large as the $i$th prime number.

Each conjugate of $C_i-Z(G)$ has
$2c_i-2$ elements, and the normalizer hypothesis tells us
there are $g/2c_i$ many conjugates.  Therefore
$(2c_i-2)g/2c_i=g\bigl(1-\frac{1}{c_i}\bigr)$ elements of
$G-Z(G)$ are conjugate into $C_i-Z(G)$.  Our partition of $G$ gives
\begin{equation}
\label{eq-class-equation}
2g=2+g\sum_{i=1}^n \textstyle(1-\frac{1}{c_i})
\end{equation}
We can rewrite this as $g(2-n) = 2-\sum_{i=1}^n{g}/{c_i}$.  Since $G$
has no index~$2$ subgroups, each term $g/c_i$ in the sum is larger than~$2$.
Since the right side 
is negative,  $g(2-n)$ is also.  So
$n>2$.

In fact $n=3$.  Otherwise, we would use
$c_1\geq2$, $c_2\geq3$, $c_3\geq5$, $c_4\geq7$ to see that the sum on the right side of
\eqref{eq-class-equation} is
$$
\hbox{(at least $\frac12$)}+ 
\hbox{(at least $\frac23$)}+ 
\hbox{(at least $\frac45$)}+ 
\hbox{(at least $\frac67$)}+\cdots
>2
$$
which is a contradiction.  Now we rewrite \eqref{eq-class-equation} as
$\frac{1}{c_1}+\frac{1}{c_2}+\frac{1}{c_3}=1+2/g$.  In particular, the
left side must be larger than~$1$, which requires $c_1=2$, $c_2=3$,
$c_3=5$.  Then
$\frac{1}{2}+\frac{1}{3}+\frac{1}{5}=1+2/g$
gives a formula for $g$,
namely $g=60$, so $|G|=120$.  So $G/Z(G)$ is nonsolvable of
order~$60$.
A Sylow's theorem exercise rules out the possibility that there are $15$ Sylow
2-subgroups, so there must be~$5$, and it follows easily that
$G/Z(G)\iso A_5$.
So $G$ has structure $2.A_5$.  Finally, $A_5$ has a unique
perfect central extension by $\Z/2$, namely the binary icosahedral
group \cite[Prop.~13.7]{Passman}.  
\end{proof}

\begin{theorem}[Burnside's transfer theorem]
Suppose $G$ is a finite group, and $P$ is a Sylow subgroup that is
central in its normalizer.  Then $P$ maps faithfully to the
abelianization $G/G'$.
\end{theorem}

\begin{proof}
\cite[Thm.~5.13]{Isaacs}, \cite[Thm.~4.3]{Gorenstein} or \cite[Thm~5.2.9]{Wolf}.
\end{proof}

\begin{corollary}[Cyclic transfer]
\label{cor-cyclic-transfer}
Suppose $G$ is a finite group, $p$ is a prime, and $G$'s Sylow
$p$-subgroups are cyclic.  If some nontrivial $p$-group is
central in its normalizer, or maps nontrivially to $G/G'$, then every
Sylow $p$-subgroup maps faithfully to $G/G'$.   In particular, this
holds  if $p$ is the smallest prime dividing $|G|$.
\end{corollary}

\begin{proof}
Suppose that $P_0\leq G$ is a $p$-group satisfying either of the two
conditions, and choose a Sylow $p$-subgroup $P$ containing it.  It is
cyclic by hypothesis, so
$P_0$ is characteristic in $P$, so $N(P)$ lies in $N(P_0)$.  The
automorphisms of $P$ with order prime to $p$ act nontrivially on every
nontrivial subquotient of~$P$.  
Under either hypothesis, $P_0$ (hence $P$) has a nontrivial
subquotient on which $N(P)$ acts trivially.  Therefore the image of $N(P)$ in $\Aut
P$ contains no elements of order prime to $p$.
It contains no
elements of order~$p$ either, since $P$ is abelian.  So $P$ is central
in its normalizer and we can apply Burnside's transfer theorem.  Since
$P$ maps faithfully to $G/G'$, so does every Sylow $p$-subgroup.

For the final statement, observe that $\Aut P$ has no 
elements of prime order${}>p$.  So its normalizer must act
trivially on it, and we can apply the previous paragraph.
\end{proof}

\begin{theorem}[Schur-Zassenhaus]
  Suppose $G$ is a group, $N$ is a normal subgroup, and $|N|$ and
  $|G/N|$ are coprime.  Then there exists a complement to $N$ in $G$,
  and all complements are conjugate.
  \qed
\end{theorem}

As stated, this relies on the odd order theorem.  But we only
need the much more elementary case that $N$ is abelian (Theorem 3.5 of \cite{Isaacs}).

In determining the structure of his groups, Wolf used a theorem of
Burnside: if all Sylow subgroups of a given group $H$ are cyclic, then
$H'$ and $H/H'$ are cyclic of coprime order, $H'$ has a complement,
and all complements are conjugate.  We prefer the following
decomposition $H= A:B$ because of its ``persistence'' property
\eqref{item-decomposition-persists}.  We only need this property for
the imperfect case (section~\ref{sec-imperfect}).

\begin{lemma}[Metacyclic decomposition]
\label{lem-metacyclic-decomposition}
Suppose $H$ is a finite group, all of whose Sylow subgroups are
cyclic.  Define $A$ as the subgroup generated by $H'$ and all of $H\!$'s
Sylow subgroups that are central.  Then $A$ has the following
two properties and is characterized by them:
\begin{enumerate}
\item
\label{item-cyclic-of-coprime-orders}
$A$ is normal, and $A$ and $H/A$ are cyclic of coprime orders.
\item
\label{item-Sylow-subgroups-of-B-act-nontrivially}
Every nontrivial Sylow subgroup of $H/A$ acts nontrivially on $A$.
\end{enumerate}
Furthermore,
\begin{enumerate}
\setcounter{enumi}{2}
\item
\label{item-complements-exist-and-are-conjugate}
$A$ has a complement $B$, and all complements are conjugate. 
\item
\label{item-decomposition-persists}
Suppose a finite group $G$ contains $H$ as a normal subgroup, with
$|G/H|$ coprime to $|H|$. Then there is a complement $C$ to $H$, such
that the decomposition $H\sim A:B$ in
\eqref{item-complements-exist-and-are-conjugate} extends to $G\sim
A:(B\times C)$.  Furthermore, all complements of $H$ that
normalize~$B$ are conjugate under 
$N_H(B)$.
\end{enumerate}
\end{lemma}

\begin{proof}
First we show that $H$ is solvable.  If $p$ is the smallest prime
dividing $|H|$, and $P$ is a Sylow $p$-subgroup, then corollary~\ref{cor-cyclic-transfer} 
shows that $H$ has a
quotient group isomorphic to $P$.  The kernel is solvable by
induction, so $H$ is too.

Now let $F$ be the Fitting subgroup of $H$, i.e., the unique
maximal normal nilpotent subgroup.  Being nilpotent, it is the product
of its Sylow subgroups.  Since these are cyclic, so is $F$.  Also,
$H/F$ acts faithfully on $F$, for otherwise $F$ would lie in a
strictly larger normal nilpotent subgroup.  As a subgroup of the
abelian group $\Aut F$, $H/F$ is abelian.  Therefore the cyclic
group $F$
contains $H'$, so $H'$ is cyclic.  

If $p$ is a prime dividing the order of $H/H'$, then
corollary~\ref{cor-cyclic-transfer} shows that every Sylow
$p$-subgroup meets $H'$ trivially.  It follows that the orders of $H'$
and $H/H'$ are coprime.  $A$ is obviously normal.  Since $A$ is the
product of $H'$ with the central Sylow subgroups of $H$, we see that
$A$ and $H/A$ also have coprime orders.  Since $A$ contains $H'$,
$H/A$ is abelian.
Having cyclic Sylow subgroups,  $H/A$ is  cyclic.  We have proven \eqref{item-cyclic-of-coprime-orders}.

Because $|A|$ and $|H/A|$ are coprime, the Schur-Zassenhaus theorem
assures us that $A$ has a complement $B$, and that all complements are
conjugate, proving \eqref{item-complements-exist-and-are-conjugate}.
For \eqref{item-Sylow-subgroups-of-B-act-nontrivially}, suppose a
Sylow subgroup of $B$ acts trivially on $A$.  Then it is central in
$H$, so $A$ contains it by definition, which is a contradiction.

Next we prove the ``persistence'' property
\eqref{item-decomposition-persists}, so we assume its hypotheses.  Since $A$ is
characteristic in $H$, it is normal in $G$.  Since all complements to
$A$ in $H$ are conjugate in $A$, the Frattini argument shows that $N_G(B)$
maps onto $G/H$.  Applying the Schur-Zassenhaus theorem to $N_H(B)$
inside $N_G(B)$ yields a complement $C$ to $H$ that normalizes $B$.
That theorem also shows that all such complements are conjugate under
$N_H(B)$

To prove \eqref{item-decomposition-persists} it remains
only to show that $B$ and $C$ commute.  If $C$
acted nontrivially on $B$, then $G'$ would contain a nontrivial
$p$-group for some prime $p$ dividing $|B|$.  Using
corollary~\ref{cor-cyclic-transfer} as before, it follows that $G'$
contains a Sylow $p$-subgroup $P$ of $B$.  Since $P$ lies in the
commutator subgroup of $G$, it must act trivially on $A$.  This
contradicts~\eqref{item-Sylow-subgroups-of-B-act-nontrivially}.  So
$B$ and $C$ commute, completing the proof
of~\eqref{item-decomposition-persists}.  (Remark: since
$N_H(B)=C_A(B)\times B$, we could replace $N_H(B)$ in the statement of
\eqref{item-decomposition-persists} by $C_A(B)$.)

All that remains is to show that 
\eqref{item-cyclic-of-coprime-orders} and
\eqref{item-Sylow-subgroups-of-B-act-nontrivially}
characterize~$A$; suppose $A^*\leq H$ has these properties.  By
\eqref{item-Sylow-subgroups-of-B-act-nontrivially}, all
the Sylow subgroups of $H$ that act trivially on $A^*$ lie in $A^*$.
Since $A^*$ is cyclic, $\Aut A^*$ is abelian, so $H'$ acts trivially
on it.  We already saw that $H'$ is the product of some of $H$'s Sylow
subgroups, so $H'$ lies in $A^*$.  The central Sylow subgroups of $H$
also act trivially on $A^*$, so also lie in $A^*$.  We have shown that
$A^*$ contains~$A$.  If $A^*$ were strictly larger than $A$, then the
coprimality of $|A^*|$ and $|H/A^*|$ would show that $A^*$ contains
a Sylow subgroup of $H$ that is not in $A$.  But then
$A^*$ is nonabelian by
property
\eqref{item-Sylow-subgroups-of-B-act-nontrivially} of
$A$, and therefore property
\eqref{item-cyclic-of-coprime-orders} fails for~$A^*$.
\end{proof}

\section{The perfect case}
\label{sec-perfect}

\noindent
In this section and the next we prove theorems \ref{thm-the-groups}
and~\ref{thm-uniqueness-of-structure} inductively.  We suppose throughout that $G$ is a finite group that acts
freely on a sphere of some dimension.  Under the assumption that every
proper subgroup has one of the structures listed in
theorem~\ref{thm-the-groups}, we will prove that $G$ also has such a
structure.  In this section we also assume $G$ is perfect.  This includes
base case $G=1$ of the induction, which occurs in case~\ref{type-I}.
The only other perfect group in theorem~\ref{thm-the-groups} is
$2A_5$.  Theorem~\ref{thm-uniqueness-of-structure} is trivial for
$G\iso1$ or $2A_5$.  Therefore it will suffice to prove 
$G\iso2A_5$ under the assumption $G\neq1$. 

\begin{lemma}
\label{lem-2-Sylow-quaternionic}
$G$'s Sylow $2$-subgroups are quaternionic, in particular nontrivial.
\end{lemma}

\begin{proof}
By lemma~\ref{lem-Sylow-subgroups-are-cyclic-or-quaternion}, all the odd Sylow subgroups are cyclic.  If the Sylow
$2$-subgroups were too, then corollary~\ref{cor-cyclic-transfer}, applied to the smallest
prime dividing $|G|$, would contradict perfectness.  
Now
lemma~\ref{lem-Sylow-subgroups-are-cyclic-or-quaternion} shows that the Sylow $2$-subgroups must be quaternionic.
\end{proof}

\begin{lemma}
\label{lem-4-elements-form-a-class}
$G$'s $4$-elements form a single conjugacy class.
\end{lemma}

\begin{proof}
Let $T$ be a Sylow $2$-subgroup (quaternionic by the previous lemma)
and $U$ a cyclic subgroup of index~$2$.   We consider the action of $G$ on the coset space $G/U$,
whose order is twice an odd number.  By lemma~\ref{lem-unique-involution}, $G$
contains a unique involution, necessarily central.
Since it lies in every conjugate of $U$, it acts trivially on $G/U$.
Now let $\phi$ be any $4$-element.  Since its square acts trivially, $\phi$ acts by exchanging some points in pairs.  Since $G$ is
perfect, $\phi$ must act by an even permutation, so the number of
these pairs is even.  Since the size of $G/U$ is not divisible by~$4$,
$\phi$ must fix some points.  The stabilizers of these points are
conjugates of $U$, so $\phi$ is conjugate into $U$.  Finally, the two
$4$-elements in $U$ are conjugate since $T$ contains an element
inverting~$U$.
\end{proof}

\begin{lemma}
\label{lem-no-odd-normal-subgroups}
$O(G)=1$.
\end{lemma}

\begin{proof}
  Suppose otherwise.  Since $G$'s odd Sylow subgroups are cyclic,
  lemma~\ref{lem-metacyclic-decomposition}\eqref{item-cyclic-of-coprime-orders} shows that $O(G)$
  has a characteristic cyclic subgroup of prime order.  Because this
  subgroup has abelian automorphism group, and $G$ is perfect, $G$
  acts trivially on it.  Now corollary~\ref{cor-cyclic-transfer} shows
  that $G$ has nontrivial abelianization, contradicting perfectness.
\end{proof}

\begin{lemma}
\label{lem-maximal-subgroups-lack-odd-central-subgroups}
Every maximal subgroup $M$ of $G$ has center of order~$2$.
\end{lemma}

\begin{proof}
First, $M$ contains $G$'s central involution.  Otherwise, adjoining it
to $M$ would yield $G$ by maximality.  Then $G$ would have an
index~$2$ subgroup, contrary to perfectness.  
Next we show that $M$ has no central subgroup $Y$ of
order~$4$; suppose it did. By the conjugacy of $\Z/4$'s in $G$, and the fact that
$G$'s Sylow $2$-subgroups are quaternionic, $N(Y)$
contains an element inverting $Y$.  So $N(Y)$ is strictly larger than
$M$, hence coincides with $G$, and the map $G\to\Aut(Y)\iso\Z/2$ is
nontrivial, contrary to perfectness.

Finally we show that $M$ has no central subgroup $Y$ of odd prime
order${}>1$.  
The previous lemma shows that $N(Y)$ is strictly smaller than
$G$.  Since $M$ normalizes $Y$ and is maximal, it is $Y$'s full
normalizer.  
Since $Y$ is central in $M$, we see that $N(Y)$ acts
trivially on~$Y$.  Now corollary~\ref{cor-cyclic-transfer} shows that
$G/G'$ is 
nontrivial, contradicting perfectness.
\end{proof}

The next lemma is where our development diverges from Wolf's.

\begin{lemma}[Maximal subgroups]
\label{lem-maximal-subgroup-structures-perfect-case}
Every maximal subgroup $M$ of $G$ has one of the following structures,
with $O(M)$ a cyclic group.
\renewcommand\theenumi{{\rm\Roman{enumi}}}
\begin{enumerate}
\item
$O(M):\hbox{\rm(cyclic 2-group of order${}>2$)}$
\item
$O(M):\hbox{\rm(quaternionic group)}$
\item
$2A_4$
\item
\leavevmode \hbox{$\bigl(O(M).2A_4\bigr)\cdot2$} where the elements of
$M$ outside  $O(M).2A_4$ act on $O(M)$ by inversion and
on the quotient $2A_4$ by outer automorphisms.
\item
$2A_5$
\item
\leavevmode \leavevmode \hbox{$\bigl(O(M)\times2A_5\bigr)\cdot2$}
where 
the elements of
$M$ outside  $O(M)\times2A_5$ act on $O(M)$ by inversion and on $2A_5$ by 
outer automorphisms.
\end{enumerate}
The $G$-normalizer of any nontrivial subgroup of $O(M)$ is~$M$.
If $M_1$ and $M_2$ are non-conjugate maximal subgroups, then
$O(M_1)$ and $O(M_2)$ have coprime orders.
\end{lemma}

\begin{proof}
By induction, $M$ has one of the structures in theorem~\ref{thm-the-groups}.  In
light of the previous lemma, we keep only those with center of
order~$2$.  Here are the details.  In every case, the prime order
elements of $B$ are central in $M$, so they cannot exist, so $B=1$.
For types~\ref{type-I}--\ref{type-II} this leaves $O(M)=A$ and establishes our
claimed structure for $M$.  
For type~\ref{type-I} we must also show that the cyclic $2$-group, call it $T$,
has order${}>2$. Otherwise, $A$ is central in $M$,
hence trivial, so $T$ is all of $M$ and has order~$2$.  That is, the
center of $G$ is a maximal subgroup of $G$, which is a contradiction
because no group can have this property.

For type~\ref{type-V} we have shown $M\sim2A_5\times A$, so $A$ is
central, hence trivial.  For type~\ref{type-III} we know $M\sim
(Q_8\times A):\Theta$, with $\Theta$'s $3$-elements centralizing~$A$.
It follows that $|\Theta|=3$, because otherwise $\Theta$'s
$3$-elements would also centralize $Q_8$, hence be central in $M$.
From $|\Theta|=3$ it follows that $A$ is central in $M$, hence
trivial.  So $M\sim Q_8:3$ with the $\Z/3$ acting nontrivially on
$Q_8$.  Since $\Aut(Q_8)\iso S_4$ has a unique class of $3$-elements, 
$M$ is determined up to isomorphism, namely $M\iso2A_4$.

For type~\ref{type-VI} we know $M\sim(2A_5\times
A)\cdot2$ and $O(M)=A$. Also, $A$ decomposes as the direct sum of its
subgroup inverted by the nontrivial element $t$ of $M/(2A_5\times
A)\iso\Z/2$ and its subgroup fixed pointwise by~$t$.  The latter
subgroup is central in $M$, hence trivial, so $t$ inverts $A$ as
claimed.  Also, $t$'s image in $\Out 2A_5$ is nontrivial by 
the definition of type~\ref{type-VI} groups. 

Finally, for type~\ref{type-IV} we have $M\sim\bigl((Q_8\times
A):\Theta\bigr)\cdot2$.   By remark~\ref{remark-implied-relations}, $\Theta$ and
$A$ commute.  So $O(M)$ is cyclic and generated by $A$ and the index~$3$ subgroup
of $\Theta$, leaving $M\sim(O(M).2A_4)\cdot2$.  By the argument
for type~\ref{type-VI}, the elements of $M$
mapping nontrivially to $\Z/2$ must invert $O(M)$ and act on $Q_8$ by
outer automorphisms. 

We have shown in each case that $O(M)$ is cyclic. So its subgroups are
characteristic in $O(M)$, hence normal in $M$.  By
lemma~\ref{lem-no-odd-normal-subgroups} and maximality, $M$ is the
full normalizer of any nontrivial subgroup of $O(M)$.
For the last claim of the theorem, suppose a prime $p$
divides the orders of $O(M_1)$ and $O(M_2)$.
We have just shown that $M_1$ is the
normalizer of a cyclic group of order~$p$, and that $M_2$ is the normalizer of
another.  These cyclic groups are $G$-conjugate, so  $M_1$
and~$M_2$ are also.
\end{proof}

\begin{lemma}[Done if $Q_{16}\not\leq G$]
  \label{lem-done-if-Q-8}
  Suppose the Sylow $2$-subgroups of $G$ have order~$8$.  Then
  $G\iso2A_5$. 
\end{lemma}

\begin{proof}
  Suppose $M$ is a maximal subgroup of $G$.  It cannot have type
  \ref{type-IV} or~\ref{type-VI}, because these contain
  copies of $Q_{16}$.  If $M$ has type \ref{type-III} or~\ref{type-V}
  then $M\iso2A_4$ or~$2A_5$ by lemma~\ref{lem-maximal-subgroup-structures-perfect-case}.  We claim that in the
  remaining cases, $M$ is binary dihedral.

  If $M$ has type~\ref{type-I} then it has structure $O(M):T$ where
  $T$ is a cyclic group of order${}\geq4$.  Since this is the largest
  a cyclic $2$-group in $G$ can be, $T$ has order exactly~$4$.  A
  generator for it must negate $O(M)$, or else $M$'s center would have
  order${}>2$.  We have shown that $M$ is binary dihedral.

  Now suppose $M$ has type~\ref{type-II}, hence structure $O(M):Q$
  where $Q\iso Q_8$.  If $O(M)=1$ then $M\iso Q_8$, which is binary
  dihedral as claimed.  So suppose $O(M)\neq1$ and let $P$ be any
  Sylow subgroup of $O(M)$.  Since $\Aut P$ is cyclic, and
  $Q/Q'\iso2\times2$, some $4$-element of $Q$ acts trivially on $P$.
  Since all $4$-elements are conjugate, the centralizer of any one of
  them has order divisible by $|P|$.  Now fix a particular
  $4$-element~$\phi$.  Letting $P$ vary over all Sylow subgroups of
  $O(M)$ shows that $C(\phi)$ has order divisible by $|O(M)|$.  The
  only maximal subgroup of $G$ that could contain a cyclic group of
  order  $4\cdot|O(M)|$ is~$M$, up to
  conjugacy.
  So after conjugation we may suppose that
  $O(M)\leq C(\phi)\leq M$.  We have shown that some $4$-element
  $\phi$ of $Q$ centralizes $O(M)$.  So $Q$ acts on $O(M)$ via a
  quotient group of order${}\leq2$.  This quotient must be $\Z/2$,
  acting by negation, because otherwise $M$ would have center larger
  than $\Z/2$.   It follows that $M$ is binary dihedral.

  We have shown that every maximal subgroup of $M$ is binary dihedral,
  binary tetrahedral or binary icosahedral.  
  By examining normalizers in these groups, one checks that every
  noncentral cyclic subgroup of $G$ has binary dihedral normalizer.  So
  $G\iso2A_5$ by lemma~\ref{lem-G-is-2A5}.
\end{proof}

To prove the perfect case of theorems \ref{thm-the-groups}
and~\ref{thm-uniqueness-of-structure}, it now suffices to rule out the case
$Q_{16}\leq G$.  We devote the rest of this section to this.

\begin{lemma}[$\Z/4$ normalizers]
  \label{lem-Z-4-normalizers}
  For any subgroup $\Phi\iso\Z/4$ of $G$, $N(\Phi)$ has structure
  $(\hbox{\rm odd group}).(\hbox{\rm Sylow $2$-subgroup of $G$})$.
\end{lemma}

\begin{proof}
  We claim first that $N(\Phi)$ has structure $(\hbox{\rm odd
    group}).(\hbox{\rm $2$-group})$.  Choosing a maximal subgroup $M$
  containing $N(\Phi)$, it suffices to show that $N_M(\Phi)$ has this
  structure.  To prove this one considers each possible structure for
  $M$ listed in lemma~\ref{lem-maximal-subgroup-structures-perfect-case}, and each subgroup $\Z/4$ of it.  To
  finish the proof we use the 
  facts that some $\Z/4$ is normal in some  Sylow
  $2$-subgroup, and that 
  all $\Z/4$'s are
  conjugate.
\end{proof}

\begin{lemma}[$Q_8$ normalizers]
  \label{lem-Q8s-if-a-Q16-exists}
  Suppose $G$ contains a copy of $Q_{16}$.  Choose $Q\leq G$
  isomorphic to $Q_8$ and write $N$ for its normalizer.  Then
  \begin{enumerate}
  \item
    \label{item-Q-8-normalizer} $N\sim\bigl(O(N).2A_4\bigr)\cdot2$.
  \item
    \label{item-condition-to-split-off-2A4} $Q$ lies in a group $2A_4$
    if and only if $3\nmid|O(N)|$.
  \item
    \label{item-more-than-one-Q-8}
    $G$ has more than one conjugacy class of $Q_8$ subgroups.
  \item
    \label{item-2A4-exists} $G$ contains a subgroup $2A_4$.
  \end{enumerate}
\end{lemma}

\begin{proof}
  \eqref{item-Q-8-normalizer}
  We begin by exhibiting some elements of~$N$.  Choose any
  subgroup $\Phi\iso\Z/4$ of $Q$.  There exists a Sylow $2$-subgroup of
  $N(\Phi)$ that contains $Q$.
  So $Q$ lies in some $Q_{16}$ that normalizes~$\Phi$.  This $Q_{16}$ contains an element that inverts $\Phi$ and
  exchanges the other two $\Z/4$ subgroups of $Q$.  We have shown that
  $N$ contains an element that normalizes our chosen $\Z/4$ and
  exchanges the other two $\Z/4$'s in~$Q$.  It follows that $N$ acts
  by $S_3$ on the three $\Z/4$ subgroups of~$Q$.
  
  Now we fix a maximal subgroup $M$ of $G$ that contains~$N$.  The
  property of~$N$ just established forces $M$ to have type
  \ref{type-IV} or~\ref{type-VI}.  That is,
  $M\sim\bigl(O(M).(\hbox{$2A_4$ or $2A_5$})\bigr)\cdot2$ with $Q$
  lying in the index~$2$ subgroup~$M'$.  In either case, $N$ coincides
  with $N_M(Q)$, which has the
  stated structure.  Also,
  $O(N)=O(M)$, which we will use later in the proof.

  \eqref{item-condition-to-split-off-2A4} ``If'' follows from the
  existence of Sylow $3$-subgroup, isomorphic to $\Z/3$, in
  $N$.  For ``only if'', suppose $3$ divides $|O(N)|$.  Then $N$
  cannot contains any $2A_4$, because all its
  $3$-elements are central in $N$.

  \eqref{item-more-than-one-Q-8} Fix $\Phi\leq Q$ isomorphic to $\Z/4$.  By lemma~\ref{lem-Z-4-normalizers}
  there is a subgroup $Q^*\iso Q_8$ of $N(\Phi)$ that is not
  $N(\Phi)$-conjugate to $Q$.  We claim that $Q^*$ is not
  $G$-conjugate to $Q$ either; suppose to the contrary that some $g\in
  G$ conjugates $Q^*$ to $Q$.  By $\Phi\leq Q^*$ we see that $g$ sends
  $\Phi$ into $Q$.  By replacing $g$ by its composition with an
  element of $N$ that sends $\Phi^g$ back to $\Phi$, we may suppose
  without loss that $g$ normalizes~$\Phi$.  That is, $Q$ and $Q^*$ are
  conjugate in $N(\Phi)$, which is a contradiction.
  
  \eqref{item-2A4-exists} Supposing that $G$ contains no $2A_4$, we will show that
  every subgroup $Q^*\iso Q_8$ of $G$ is conjugate to $Q$, which
  contradicts \eqref{item-more-than-one-Q-8}.  Applying the argument for \eqref{item-Q-8-normalizer} to $Q^*$,
  we write $N^*$ for its normalizer and $M^*$ for a maximal subgroup
  containing~$N^*$.  By \eqref{item-condition-to-split-off-2A4} and the nonexistence of $2A_4$'s, $3$
  divides the orders of $O(N)$ and $O(N^*)$.  These groups are the same as
  $O(M)$ and $O(M^*)$.  Since $|O(M)|$ and $|O(M^*)|$ have a common
  factor,
  the last part of lemma~\ref{lem-maximal-subgroup-structures-perfect-case} shows that $M$ and $M^*$ are
  conjugate.
  So their index two subgroups $M'$ and ${M^*}'$ are
  conjugate, both of which   have structure
  $(\hbox{odd group}).(\hbox{$2A_4$ or $2A_5$})$.
  Since $Q$
  resp.\ $Q^*$ is a Sylow $2$-subgroup of $M'$ resp.\ ${M^*}'$, it
  follows 
  that $Q$ and $Q^*$ are conjugate.
\end{proof}

\begin{lemma}[Free actions of binary dihedral groups]
  \label{lem-free-actions-of-binary-dihedral-groups}
  Suppose $H$ is a binary dihedral group, $U$ is an irreducible
  fixed-point-free real representation of~$H$, and $\a$ is the natural
  map $\R[H]\to\End(U)$.  
  Then
  $\a(\R[H])\iso\H$.  Furthermore, if $J$ is
  any binary dihedral subgroup of $H$, then $\a(\R[J])=\a(\R[H])$.
\end{lemma}

\begin{proof}
  It is easy to see that the lemma holds for the representations
  \eqref{eq-binary-dihedral-representation}, using their description in terms of $\H$.  So it will
  suffice to show that $U$ is one of them.  Let $C$ be an index~$2$
  cyclic subgroup of $H$.  Under~$C$, $U$ decomposes as a direct sum
  of $2$-dimensional spaces, on each of which $C$ acts faithfully by
  rotations.  Fix one, say~$T$, and consider the induced
  representation $\Ind_C^H(T)$, of dimension~$4$.  Its canonical image in $U$ is
  $H$-invariant, hence all of~$U$.  This image is larger than $2$-dimensional,
  because a binary dihedral group has no $2$-dimensional free real
  representations. (Every finite subgroup of ${\rm O}(2)$ is cyclic or dihedral.)
  Therefore $V\iso\Ind_C^H(T)$.  And the
  representations in \eqref{eq-binary-dihedral-representation} are exactly the $H$-representations of
  this form.
\end{proof}

\begin{lemma}[Free actions of the binary tetrahedral group]
\label{lem-free-actions-of-binary-tetrahedral-group}
Assume $A\iso2A_4$ acts freely on a real vector space~$V$, and write
$Q$ for the copy of $Q_8$ in~$A$.  Then the image of $\R[A]$ in
$\End(V)$ is isomorphic to $\H$ and equal to the image of $\R[Q]$.  In
particular, every $Q$-invariant subspace is also $A$-invariant.
\end{lemma}

\begin{proof}
  Fix an identification of $A$ with
  \begin{equation}
    \label{eq-concrete-model-of-2A4}
    2A_4=\bigl\{\pm1,\pm i,\pm j,\pm k, (\pm1\pm i\pm j\pm
    k)/2\bigr\}\sset\H^*
  \end{equation}
  We claim that $A$ has a unique irreducible free real
  representation.  It follows that $V$ is a
  direct sum of copies of $A$'s left-multiplication action on $\H$, which proves the lemma.

  For the claim, fix an irreducible free $A$-module $U$.  Choose a
  $Q$-sub\-mod\-ule $T$ which is $Q$-irreducible.  In the proof of
  lemma~\ref{lem-free-actions-of-binary-dihedral-groups} we showed that \eqref{eq-binary-dihedral-representation} accounts for all irreducible
  free actions of binary dihedral groups. For $Q_8$, there was only
  one.  So we may identify $T$ with the real vector space underlying
  $\H$, with $Q=\{\pm1,\pm i,\pm j,\pm k\}\sset A$ acting by left
  multiplication.  The image of the natural map $\Ind_Q^A(T)\to U$ is
  $A$-invariant, hence all of~$U$.  By definition, $\Ind_Q^A(T)$ is
  the real vector space underlying $\H^3$, with $i\in Q$ acting by
  $(x,y,z)\mapsto(ix,jy,kz)$, similarly with $i,j,k$ cyclically
  permuted, and $\theta:=(-1+i+j+k)/2\in A$ acting by
  $(x,y,z)\mapsto(z,x,y)$.  Obviously the kernel of $\Ind_Q^A(T)\to U$
  contains the fixed points of the nonidentity elements of~$A$.  For
  $\theta$ and $i\circ\theta$ these are $\{(x,x,x)\}$ and
  $\{(x,jx,-ix)\}$, which together span an
  $8$-dimensional subspace, hence the whole kernel.  This proves the
  uniqueness of~$U$: it is the quotient of $\Ind_Q^A(T)$ by the
  subspace generated by the fixed points of the nonidentity elements
  of~$A$.  
\end{proof}

\begin{lemma}[The final contradiction]
\label{lem-trick-for-large-Sylow-2-subgroups}
$G$ contains no subgroup $Q_{16}$.
\end{lemma}

\begin{proof}
  Suppose otherwise, and fix an irreducible fixed-point free real
  representation $V$ of $G$.  We write $\a$ for the natural map
  $\R[G]\to\End V$.

  By lemma~\ref{lem-Q8s-if-a-Q16-exists}\eqref{item-2A4-exists}, $G$ has a subgroup $A^*\iso2A_4$.  We write
  $Q^*$ for its $Q_8$ subgroup.  We fix a $Q_{16}$-subgroup of $G$
  that contains $Q^*$.  This is the only group isomorphic to $Q_{16}$ we
  will consider, so we just write $Q_{16}$ for it.  Write $Q$ for the
  other $Q_8$ subgroup of $Q_{16}$.  We will distinguish two cases,
  according to whether there exists some $A\iso 2A_4$ containing
  $Q$.  In each case we will prove $\a(\R[G])\iso\H$, which implies that $G$ is
  isomorphic to a subgroup of $\H^*$.  We assume this for the moment.

  A well-known property of $\H^*$ is that every non-central cyclic
  subgroup has $\H^*$-normalizer isomorphic to the ``continuous binary
  dihedral group'', meaning the nonsplit extension $(\R/\Z)\cdot2$.
  It follows
  that the $G$-normalizer of any non-central cyclic subgroup is 
  cyclic or binary dihedral.  The cyclic case is ruled out by
  corollary~\ref{cor-cyclic-transfer} and the perfectness of~$G$.  So lemma~\ref{lem-G-is-2A5}  applies,
  proving $G\iso2A_5$.

  It remains to prove $\a(\R[G])\iso\H$.  First suppose that $A$
  exists.  $G$~is generated by $A$ and $A^*$, because no maximal
  subgroup contains two copies of $2A_4$, whose $Q_8$ subgroups
  generate a copy of $Q_{16}$.  Now fix a irreducible
  $Q_{16}$-submodule $U$ of $V$.  It is obviously $Q$- and
  $Q^*$-invariant, hence $A$- and $A^*$-invariant by
  lemma~\ref{lem-free-actions-of-binary-tetrahedral-group}.  So it is
  $G$-invariant by $\gend{A,A^*}=G$, and the $G$-irreducibility of $V$
  implies $U=V$.  Lemmas
  \ref{lem-free-actions-of-binary-dihedral-groups}
  and~\ref{lem-free-actions-of-binary-tetrahedral-group} also give us
  the equalities
  \begin{equation}
    \label{eq-the-trick-first-version}
    \a(\R[A])=\a(\R[Q])=\a(\R[Q_{16}])=\a(\R[Q^*])=\a(\R[A^*])
  \end{equation}
  and say that this subalgebra of $\End(U)$ is isomorphic to~$\H$.  
  Since $A$ and $A^*$ generate~$G$, $\a(\R[G])$ also equals this copy of $\H$, finishing the proof in the case that $A$ exists.

  Now suppose $Q$ lies in no copy of~$2A_4$, and set $N=N_G(Q)$ and
  $\Phi=Q\cap Q^*\iso\Z/4$.  By \eqref{item-Q-8-normalizer}
  and~\eqref{item-condition-to-split-off-2A4} of
  lemma~\ref{lem-Q8s-if-a-Q16-exists} we have $N\sim(O(N).2A_4)\cdot2$
  with $|O(N)|$ divisible by~$3$.  By construction $Q^*$ normalizes
  $Q$, and by $Q^*\neq Q$ we see that $Q^*$ contains a $4$-element
  $q^*$ which lies in $N$ but outside its index~$2$ subgroup
  $O(N).2A_4$.  So $q^*$ inverts $O(N)$.  Of course, $q^*$ also inverts
  $\Phi$.  And $O(N)$ commutes with $Q$ (hence $\Phi$) by the
  known structure of~$N$.  So $H=\gend{q^*,O(N),\Phi}=\gend{Q^*,O(N)}$
  is binary dihedral.  The rest of the argument is similar to the
  previous case.

  Namely, we have $G=\gend{H,A^*}$ because
  no maximal subgroup $M$ of $G$ contains a copy of $Q_8$ which
  normalizes one copy of $\Z/3$ in $M$ and is normalized by a
  different copy of $\Z/3$ in $M$.
  Now we choose an irreducible $H$-submodule $U$ of $V$.  It is
  trivially $Q^*$-invariant.  Then
  lemma~\ref{lem-free-actions-of-binary-tetrahedral-group} shows that
  $U$ is $A^*$-invariant, hence $G$-invariant,
  hence all of~$V$.  Lemmas
  \ref{lem-free-actions-of-binary-dihedral-groups}--\ref{lem-free-actions-of-binary-tetrahedral-group}
  give the equalities
  \begin{equation}
  \label{eq-the-trick-second-version}
  \a(\R[H])=\a(\R[Q^*])=\a(\R[A^*])
  \end{equation}
  and say that this subalgebra of $\End(U)$ is isomorphic to~$\H$.
  By $G=\gend{H,A^*}$, this  subalgebra is also $\a(\R[G])$, finishing the proof.
\end{proof}

\begin{remark}[Frobenius complements]
  \label{remark-Frobenius-complements}
  As mentioned in the introduction, our arguments adapt easily to
  classify the Frobenius complements, which are the same groups.
  Supposing that $G$ acts freely on a vector space over a finite field
  $\F_q$, where $q$ is necessarily odd, we extend scalars to suppose
  without loss that all elements of $G$ are diagonalizable.  Then
  lemmas \ref{lem-free-actions-of-binary-dihedral-groups}--\ref{lem-trick-for-large-Sylow-2-subgroups} and their proofs still apply with $\R$
  replaced by $\F_q$, $\H$ (the algebra) by $M_2\F_q$, $\H$ (the
  module) by $\F_q^2$, and $\H^*$ by $\SL_2\F_q$.  (Although the proof
  of lemma~\ref{lem-free-actions-of-binary-tetrahedral-group} looks $\H$-specific, one defines the Hurwitz
  integers as the $\Z$-span of the quaternions
  \eqref{eq-concrete-model-of-2A4}, and tensoring with $\F_q$ gives
  $M_2\F_q$.)
\end{remark}

\section{The imperfect case}
\label{sec-imperfect}

\noindent
In this section we will complete the proofs of
theorems \ref{thm-the-groups} and~\ref{thm-uniqueness-of-structure}.  We suppose throughout that $G$ is an
imperfect finite group that acts freely and isometrically on a sphere of some dimension,
and that every proper subgroup has one of the structures listed in the
statement of theorem~\ref{thm-the-groups}.  We will prove that $G$ also has one of
these structures, in fact a unique one in the sense of
theorem~\ref{thm-uniqueness-of-structure}.  We will prove theorems
\ref{thm-the-groups} and~\ref{thm-uniqueness-of-structure} in three
special cases (which don't actually use imperfectness), and then argue that these
cases are enough.

\begin{lemma}
\label{lem-if-G-mod-odd-subgroup-is-a-2-group}
Suppose $G/O(G)$ is a $2$-group.  Then $G$ has type
\ref{type-I} or~\ref{type-II} from theorem~\ref{thm-the-groups}, for 
a
unique-up-to-conjugation triple of subgroups $(A,B,\discretionary{}{}{}T)$.
\end{lemma}

\begin{proof}
We will construct $A$, $B$ and $T$ such that $G\sim A:(B\times T)$ as
in theorem~\ref{thm-the-groups}, along the way observing
that the construction is essentially unique.  Obviously we must choose
$A$ and $B$ such that $O(G)=A:B$.  Applying
lemma~\ref{lem-metacyclic-decomposition} to $O(G)$ proves
the following.  The required coprimality of $|A|$ and $|B|$, together
with the requirement that each nontrivial Sylow subgroup of $B$ acts
nontrivially on $A$, can be satisfied in a unique way.  That is, $A$
is uniquely determined, and $B$ is determined up to conjugacy in
$O(G)$.  Using the coprimality of $|O(G)|$ and $|G/O(G)|\iso T$, part
\eqref{item-decomposition-persists} of lemma~\ref{lem-metacyclic-decomposition} shows
that the decomposition $O(G)=A:B$ extends to a decomposition
$G=A:(B\times T)$ where $T$ is determined uniquely up to conjugacy in
$N_G(B)$.

This finishes the proof of uniqueness; it remains to check a few
assertions of theorem~\ref{thm-the-groups}.
Lemma~\ref{lem-Sylow-subgroups-are-cyclic-or-quaternion} says that $T$
is cyclic or quaternionic.
Lemma~\ref{lem-abelian-subgroups-are-cyclic} assures us that every
prime-order element of $B$ or $T$ acts trivially on every prime-order
subgroup of $A$.  An automorphism of the cyclic group $A$, of order
prime to $|A|$ and
acting trivially on every subgroup of prime order, must act trivially on
all of~$A$.  So the prime-order elements of $B$ and $T$ centralize~$A$.
\end{proof}

\begin{lemma}
\label{lem-if-G-has-normal-binary-icosahedral}
Suppose $G$ contains a normal subgroup $2A_5$.  Then $G$ has
type \ref{type-V} or~\ref{type-VI} from theorem~\ref{thm-the-groups}, for a
unique-up-to-conjugation tuple of subgroups $(A,B,2A_5)$ or
$(A,B,2A_5,\Phi)$.
\end{lemma}

\begin{proof}
We will show that there is a such a tuple, unique up to conjugation,
whose $2A_5$ term is the given normal subgroup.  The existence part of
this assertion shows that $G\sim 2A_5\times(A:B)$ or
$G\sim\bigl(2A_5\times(A:B)\bigr)\cdot2$ as in
theorem~\ref{thm-the-groups}.  It follows that $G$ has a unique normal
subgroup isomorphic to $2A_5$.  So every tuple has
this particular subgroup as its $2A_5$ term.  The uniqueness of the
tuple up to conjugacy follows.

It remains to prove the lemma for the given $2A_5$ subgroup.
We define $I$ as the subgroup of $G$ acting on $2A_5$ by inner
automorphisms.  Since $\Out(2A_5)=2$, $I$ has index $1$ or~$2$ in
$G$.  By its definition, $I$ is generated by $2A_5$ and $C(2A_5)$,
whose intersection is the group $Z$ generated by $G$'s central
involution.  We claim that $Z$ is the full Sylow $2$-subgroup of
$C(2A_5)$.  Otherwise, $2A_5/Z\times \bigl(C(2A_5)/Z\bigr)\leq G/Z$
would contain an elementary abelian $2$-group of rank~$3$.  This is
impossible because the Sylow $2$-subgroups of $G/Z$ are dihedral.  From
corollary~\ref{cor-cyclic-transfer} and the fact that $C(2A_5)$ has cyclic Sylow
$2$-subgroups, we get $C(2A_5)=Z\times O(C(2A_5))$.  It follows
that $I=2A_5\times O(I)$.

Obviously we must choose $A$ and $B$ such that $A:B=O(I)$.  As in the
previous proof, lemma~\ref{lem-metacyclic-decomposition} shows that there is an essentially unique way to
satisfy the conditions that $|A|$ and $|B|$ are coprime and that every
Sylow subgroup of $B$ acts nontrivially on $A$.  That is, $A$ is uniquely
determined and $B$ is determined up to conjugacy in $O(I)$.  Also, $|A|$ and $|B|$ are
coprime to $15$ because $G$ has no subgroup $3\times3$ or $5\times5$
(lemma~\ref{lem-abelian-subgroups-are-cyclic}).  The prime-order elements of $B$ act trivially on $A$
by the same argument as in the previous proof.

We have shown that $I$ has type~\ref{type-V}, so if $I=G$ then the
proof is complete.  Otherwise, we know
$G\sim\bigl(2A_5\times(A:B)\bigr)\cdot2$ and we must construct a
suitable subgroup $\Phi\iso\Z/4$ of~$G$.  Because all complements to
$A$ in $A:B$ are conjugate, the Frattini argument shows that $N(B)$
surjects to $G/A$.  So $N(B)$ contains a Sylow $2$-subgroup $T\iso
Q_{16}$ of $G$.  By Sylow's theorem it is unique up to conjugacy in
$N(B)$.  Obviously $T\cap2A_5$ is isomorphic to $Q_8$.  Now, $T\iso
Q_{16}$ contains exactly two subgroups isomorphic to $\Z/4$ that lie
outside $T\cap2A_5$.  So we must take $\Phi$ to be one of them.  They are
conjugate under $T\cap2A_5$, so $\Phi$ is uniquely defined up to a conjugation that
preserves each of $A$, $B$ and $2A_5$.  It remains only to check that
$\Phi$ has the properties required for $G$ to be a type~\ref{type-VI}
group.  It acts on $2A_5$ by an outer automorphism because it does not
lie in $I$, by construction.     To see  that $\Phi$ commutes with
$B$, suppose to the contrary.  Then some  subgroup of $B$ of prime order~$p$
would lie in $G'$.  Then corollary~\ref{cor-cyclic-transfer} would show that the Sylow
$p$-subgroup of $B$ lies in $G'$.  So it acts trivially on $A$, which
is a contradiction.
\end{proof}

\begin{lemma}
\label{lem-if-G-has-normal-Q8-that-lies-in-commutator-subgroup}
Suppose $G$ contains a normal subgroup $Q\iso Q_8$ that lies in $G'$.  Then
$G$ has type \ref{type-III} or~\ref{type-IV} from theorem~\ref{thm-the-groups}, for a
unique-up-to-conjugation tuple of subgroups $(A,B,Q_8,\Theta)$ or
$(A,B,Q_8,\discretionary{}{}{}\Theta,\Phi)$.
\end{lemma}

\begin{proof}
Mimicking the previous proof, we will show that there is a
unique-up-to-conjugation tuple of such subgroups whose $Q_8$ term
is~$Q$.  In particular, $G\sim(Q_8\times A):(B\times\Theta)$ or
$G\sim\bigl((Q_8\times A):(B\times\Theta)\bigr)\cdot2$ as in
theorem~\ref{thm-the-groups}.  This shows that $G$ has a unique normal
subgroup isomorphic to $Q_8$, namely~$Q$, which will complete the proof for the same
reason as before.

To show that there is a unique-up-conjugation tuple of subgroups whose
$Q_8$ term is $Q$, consider the natural map $G\to\Aut Q\iso S_4$.  We
claim the image $\Gbar$ contains a $3$-element.  Otherwise,
$\Gbar\sset\Aut Q\iso S_4$ would lie in a Sylow $2$-subgroup of
$\Aut Q$, which is dihedral of order~$8$ with
commutator subgroup of order~$2$.  Therefore $|\Gbar'|$ would
have order${}\leq2$.  But this is a contradiction because $Q$ lies in
$G'$ and $\Qbar\iso2\times2$.
We have shown that the nontrivial $3$-subgroup of $\Out Q\iso S_3$ lies in
the image of $G$.  (Now we can
discard~$\Gbar$.)

Let $J$ be the preimage in $G$ of this copy of $\Z/3$.  So $G\sim
J.\hbox{($1$ or $2$)}$.  Obviously we must choose $A$, $B$ and
$\Theta$ so that $J=(Q\times A):(B\times\Theta)$.  By construction,
$J$ has a nontrivial map to $\Z/3$.  So
corollary~\ref{cor-cyclic-transfer} assures us that $J$'s Sylow
$3$-subgroups map faithfully to $J$'s abelianization.  In particular,
$J=I.\hbox{(cyclic $3$-group)}$ where $|I|$ is prime to~$3$.
Mimicking the
previous proof shows that $I=Q\times O(I)$.  
So we must choose $A$ and $B$ so that $A:B=O(I)$.

Continuing to follow the previous proof shows that there is an
essentially unique way to satisfy the conditions that $|A|$ and $|B|$
are coprime and that every Sylow subgroup of $B$ acts nontrivially on
$A$.  That is, $A$ is uniquely determined and $B$ is determined up to
conjugacy in $O(I)$.  The prime-order elements of $B$ act trivially on $A$ for
the same reason as before.    We have shown
\begin{equation*}
  G\sim
  \underbrace{
    \overbrace{\bigl(Q\times(A:B)\bigr)}^I
    .(\hbox{nontrivial cyclic $3$-group})}_J.(\hbox{$1$ or $2$})
\end{equation*}
Having worked our way to the ``middle'' of $G$, we now work outwards
and construct $\Theta$ and (if required) $\Phi$.  Using the conjugacy
of complements to $A$ in $A:B$, the Frattini argument shows that
$N_J(B)$ maps onto $J/(A:B)$.  So $N_J(B)$ contains a Sylow
$3$-subgroup of $J$, indeed a unique one up to conjugacy in $N_J(B)$.
So there is an essentially unique possibility for $\Theta$.  As
observed above, $\Theta$ acts nontrivially on $Q$.  Its $3$-elements
act trivially on $A$ for the same reason that $B$'s prime-order
elements do.  Finally, $\Theta$ commutes with $B$.  Otherwise, some
Sylow subgroup of $B$ would lie in
$G'$,  hence act trivially on
$A$, contrary to $B$'s construction.

We have shown that $J$ has type~\ref{type-III}, so if $J=G$ then we are done.
Otherwise, mimicking the previous proof shows that there exists   a
group
$\Phi\iso\Z/4$ in $N(B\times\Theta)$ that does not lie in $J$, and that such
a group is unique up to conjugacy in $N(B\times\Theta)$.
Also as before, $\Phi$ commutes with $B$.  
By $\Phi\not\leq J$, $\Phi$'s $4$-elements act on $Q$ by outer automorphisms.  Since the
images of $\Phi$ and $\Theta$ in $\Out Q\iso S_3$ do not commute, $\Phi$
cannot commute with $\Theta$.  Therefore $\Phi$'s $4$-elements must invert
$\Theta$.  So $G$ has type~\ref{type-IV} and the proof is complete.
\end{proof}

To finish the proofs of theorems \ref{thm-the-groups}
and~\ref{thm-uniqueness-of-structure} we will show that $G$ satisfies
the hypotheses of one of lemmas \ref{lem-if-G-mod-odd-subgroup-is-a-2-group}--\ref{lem-if-G-has-normal-Q8-that-lies-in-commutator-subgroup}.  By imperfectness, $G$
has a normal subgroup $M$ of prime index~$p$.  By induction, $M$ has
one of the structures \ref{type-I}--\ref{type-VI}.  We will examine
the various cases and see that one of the lemmas applies.

If $M$ has type \ref{type-V} or~\ref{type-VI} then it contains a unique
subgroup $2A_5$, which is therefore normal in $G$, so lemma~\ref{lem-if-G-has-normal-binary-icosahedral}
applies.  If $M$ has type \ref{type-III} or~\ref{type-IV} then it contains a
unique normal subgroup $Q_8$, which is therefore normal in~$G$.  Also,
this $Q_8$ lies in $M'$, hence $G'$, so lemma~\ref{lem-if-G-has-normal-Q8-that-lies-in-commutator-subgroup} applies.
Finally, suppose $M$ has type \ref{type-I} or~\ref{type-II}, so $M\sim
A:\bigl(B\times\hbox{(2-group $T$)}\bigr)$.  If $p=2$ then $O(G)=O(M)=A:B$
and $G/O(G)$ is a $2$-group, so lemma~\ref{lem-if-G-mod-odd-subgroup-is-a-2-group} applies.  

So suppose
$p>2$, and observe that $G/O(M)$ has structure $T.p$.  Choose a
subgroup $P$ of order~$p$ in $G/O(M)$.  If it acts trivially on $T$
then we have $G/O(M)\iso T\times p$, so $G/O(G)\iso T$ and
lemma~\ref{lem-if-G-mod-odd-subgroup-is-a-2-group} applies.  So
suppose $P$ acts nontrivially on $T$.  The
automorphism group of any cyclic or quaternionic $2$-group is a
$2$-group, except for $\Aut Q_8\iso S_4$.  Since $P$ acts
nontrivially, we must have $T\iso Q_8$ and $p=3$.  The nontriviality
of $P$'s action also implies $T<G'$.  From this and the fact that
$\Aut A$ is abelian, it follows that $T$ acts trivially on $A$.  So
$T$ is normal in $M$.  As $M$'s unique Sylow $2$-subgroup, it is
normal in $G$. So lemma~\ref{lem-if-G-has-normal-Q8-that-lies-in-commutator-subgroup} applies.  
This completes the proofs of
theorems \ref{thm-the-groups} and~\ref{thm-uniqueness-of-structure}.

\section{Irredundant enumeration}
\label{sec-constructive-classification}

\noindent
The uniqueness expressed in theorem~\ref{thm-uniqueness-of-structure}
makes the isomorphism classification of groups in
theorem~\ref{thm-the-groups} fairly simple.  Namely, one specifies
such a group $G$ up to isomorphism by choosing its type
\ref{type-I}--\ref{type-VI}, a suitable
number $a$ for the order of $A$, a suitable subgroup $\Gbar$ of the
unit group $(\Z/a)^*$ of the ring $\Z/a$, and a suitable number $g$ for
the order of $G$.  In a special case one must also specify a
suitable subgroup $\Gbar_0$ of $\Gbar$.  Before developing this, we
illustrate some redundancy in Wolf's list.

\begin{example}[Duplication]
\label{example-Wolf-duplication}
Wolf's presentations of type II in \cite[theorem~6.1.11]{Wolf} have
generators $A$, $B$, $R$ and relations
\begin{gather*}
A^m=B^n=1
\qquad
B A B^{-1}=A^r
\\
R^2=B^{n/2}
\qquad
R A R^{-1}=A^l
\qquad
R B R^{-1}=B^k
\end{gather*}
where $(m,n,r,k,l)$ are numerical parameters satisfying nine
conditions.  It turns out that the six choices
$(3,20,-1,-1,\pm1)$,
$(5,12,-1,-1,\discretionary{}{}{}\pm1)$ and
$(15,4,-1,-1,\hbox{$4$ or $11$})$ give isomorphic groups,  namely
$(3\times5):Q_8$.  Here one class of $4$-elements
in $Q_8$
inverts just the $\Z/3$ factor, another class 
inverts just the $\Z/5$ factor, and the third class
inverts both.  Wolf decomposes this group by choosing an index~$2$
subgroup (call it $H$) with cyclic Sylow $2$-subgroup, taking $A$ to
generate $H'$, and taking $B$ to generate a complement to $H'$ in $H$
(which always exists).  Then he takes $R$ to be a $4$-element outside
of $H$, that normalizes $\gend{B}$.  There are three ways to choose
$H$, corresponding to the three $\Z/4$'s in $Q_8$.  After $H$ is
chosen, there are unique choices for $\gend{A}$ and $\gend{B}$, but
two choices for $\gend{R}$.  These choices lead to the six different
presentations.  Our form of this group is
$A:(B\times Q_8)=15:(1\times Q_8)=15:Q_8$. (Our $A$ and $B$ are
subgroups, while Wolf's $A$ and $B$ are elements.)

A minor additional source of redundancy is that replacing Wolf's $B$ by a different
generator of $\gend{B}$ can change the parameter $r$ in the
presentation above.
\end{example}

Given a group $G$ from theorem~\ref{thm-the-groups}, we record the
following invariants.  First we record its type
\ref{type-I}--\ref{type-VI}, which is
well-defined by the easy part of
theorem~\ref{thm-uniqueness-of-structure}, proven in
section~\ref{sec-groups}.  Second we record $g:=|G|$, $a:=|A|$
and
$\Gbar$, where bars will indicate images in $\Aut A$.
Finally, and only if $G$ has
type~\ref{type-II} and its order is divisible by~$16$, we
record $\Gbar_0$, where $G_0$ is the unique index~$2$ subgroup of $G$
with cyclic Sylow $2$-subgroups.

$\Aut A$ is canonically isomorphic to the group of units $(\Z/a)^*$ of
the ring $\Z/a$, with
$u\in(\Z/a)^*$ corresponding to the $u$th power map.  Therefore we will
regard $\Gbar$ and $\Gbar_0$ as subgroups of $(\Z/a)^*$.  This
is useful when comparing two groups $G$,~$G^*$:

\begin{theorem}[Isomorphism recognition]
\label{thm-isomorphism-detection}
Two finite groups $G$ and $G^*$, that  act freely and isometrically on spheres of some
dimensions, are isomorphic if and only
if
they have
the same type, and $g=g^*$, $a=a^*$, 
$\Gbar=\Gbar^*$, and (if they are
defined) $\Gbar_0=\Gbar_0^*$.
\end{theorem}

\begin{proof}
First suppose $G\iso G^*$.  We have already mentioned that they have the same type, and they obviously have the same order.  By
the uniqueness of $A$ and $A^*$ (theorem~\ref{thm-uniqueness-of-structure}), any
isomorphism $G\to G^*$ identifies $A$ with $A^*$.  In particular,
$a=a^*$, and the action of $G$ on $A$ corresponds to that of
$G^*$ on $A^*$.   So we must have 
$\Gbar=\Gbar^*$, and (when defined) 
$\Gbar_0=\Gbar_0^*$.

Now suppose $G$ and $G^*$ have the same invariants; we must construct
an isomorphism between them.  Case analysis seems
unavoidable, but the ideas are uniform and only the details
vary. 

{\it Type~\ref{type-I}}.  Since $A$ and $A^*$ are cyclic of the same order,
we may choose an isomorphism $A\iso A^*$.  By order considerations we
have $|B|=|B^*|$ and $|T|=|T^*|$.  Again by order considerations, the
identification of $\Gbar$ with $\Gbar^*$ under the canonical
isomorphism
$\Aut A=\Aut A^*$
identifies $\Bbar$ with $\Bbar^*$ and $\Tbar$ with $\Tbar^*$.  Because
$B$ and $B^*$ are cyclic of the same order, we may lift the
identification $\Bbar\iso\Bbar^*$ to an isomorphism $B\iso B^*$, and
similarly for $T$ and $T^*$.  Together with $A\iso A^*$, these give an isomorphism $G\iso
G^*$.

{\it Type~\ref{type-II} when $16$ divides $|G|=|G^*|$}.  Because $\Gbar_0$ is
identified with $\Gbar^*_0$, the previous case gives an isomorphism
$G_0\iso G^*_0$ that identifies $A$ with $A^*$, $B$ with $B^*$ and the
cyclic $2$-group $T_0:=T\cap G_0$ with $T^*_0:=T^*\cap G^*_0$.
We will extend this to an isomorphism $G\iso G^*$.
Choose
an element $\phi$ of $T-T_0$.  By the identification of $\Tbar$ with
$\Tbar^*$ and $\Tbar_0$ with $\Tbar^*_0$, there exists an element
$\phi^*$ of $T^*-T^*_0$ whose action on $A^*$ corresponds to $\phi$'s action on
$A$.
As elements of $T-T_0$ and $T^*-T_0^*$,  $\phi$ and $\phi^*$
have order~$4$.  
They commute with
$B$ and $B^*$ respectively. Their squares are
the unique involutions in $T_0$ and $T^*_0$, which we have already identified
with each other.
So identifying $\phi$ with $\phi^*$ extends our isomorphism
$\Gbar_0\iso\Gbar^*_0$ to $\Gbar\iso\Gbar^*$.

{\it Type~\ref{type-II} when $16$ does not divide $|G|=|G^*|$}.  The argument
for type~\ref{type-I} identifies $A$ with $A^*$ and $B$ with $B^*$.  Both
$G$ and $G^*$ have Sylow $2$-subgroups isomorphic to~$Q_8$.  
$\Tbar$
is the Sylow $2$-subgroup of $\Gbar$, and is elementary abelian of
rank${}\leq2$ because 
$Q_8/Q_8'\iso2\times2$.
The identification of $\Gbar$ with $\Gbar^*$ identifies $\Tbar$
with  $\Tbar^*$.  Because $\Aut Q_8$ acts as
$S_3$ on $Q_8/Q_8'$, it is possible to lift the identification
$\Tbar\iso\Tbar^*$ to an isomorphism $T\iso
T^*$.  Now
our identifications $A\iso A^*$, $B\iso B^*$, $T\iso T^*$ fit together
to give an
isomorphism $G\iso G^*$.

{\it Type~\ref{type-III}}.  Choose an isomorphism $A\iso A^*$.  By $|G|=|G^*|$
we get $|B|=|B^*|$ and $|\Theta|=|\Theta^*|$.  The
identification of $\Gbar\leq\Aut A$ with $\Gbar^*\leq\Aut A^*$
identifies $\Thetabar$ with $\Thetabar^*$ and $\Bbar$ with $\Bbar^*$.
These identifications can be lifted to isomorphisms
$\Theta\iso\Theta^*$ and $B\iso B^*$ by the same reasoning as before.
Because all $3$-elements in $\Aut Q_8\iso S_4$ are conjugate, it is
possible to choose an isomorphism $Q_8\iso Q_8^*$ compatible with
our isomorphism $\Theta\iso\Theta^*$ and the
homomorphisms $\Theta\to\Aut Q_8$ and $\Theta^*\to\Aut Q_8^*$.  Now our
identifications $A\iso A^*$, $B\iso B^*$, $\Theta\iso\Theta^*$ and
$Q_8\iso Q_8^*$ fit together to give an isomorphism $G\iso G^*$.

{\it Type~\ref{type-IV}.}  Identify $A$, $B$ and $\Theta$ with $A^*$,
$B^*$ and $\Theta^*$ as in the previous case.  Choose generators
$\phi$ and $\phi^*$ for $\Phi$ and $\Phi^*$.  Their actions on $A$ and
$A^*$ correspond, because they act by the unique involutions in
$\Gbar$ and $\Gbar^*$ if these exist, and trivially otherwise.  The
image of $\phi$ in $\Aut Q_8\iso S_4$ normalizes the image of
$\Theta$, so together they generate a copy of $S_3$, and similarly for
their starred versions.  Any isomorphism from one $S_3$ in $\Aut Q_8$ to
another is induced by some conjugation in $\Aut Q_8$.  (One checks
this using $\Aut Q_8\iso S_4$.)  Therefore it is possible to identify
$Q_8$ with $Q_8^*$, such that the actions of $\Theta$ and $\Theta^*$
on them correspond, and the actions of $\phi$ and
$\phi^*$ also correspond.  Using this, and identifying $\phi$ with
$\phi^*$, gives an isomorphism $G\iso G^*$.

{\it Type~\ref{type-V}.} Identify $A$ and $B$ with $A^*$ and $B^*$ as in
previous cases, and $2A_5$ with $2A_5^*$ however one likes.

{\it Type~\ref{type-VI}.} Identify $A$ and $B$ with $A^*$ and $B^*$ as
before, and choose generators $\phi$ and $\phi^*$ for $\Phi$ and
$\Phi^*$.  As in the type~\ref{type-IV} case, their actions on $A$ and $A^*$
correspond.  Next, $\phi$ and $\phi^*$ act on $2A_5$ and $2A_5^*$
by involutions which are
not inner automorphisms.  All such automorphisms of $2A_5$ are
conjugate in $\Aut(2A_5)$.  (They correspond to the involutions in $S_5-A_5$.)  So we
may identify $2A_5$ with $2A_5^*$ in such a way that the actions of
$\phi$ and $\phi^*$ on them correspond.  Using this, and identifying
$\phi$ with $\phi^*$, gives an isomorphism $G\iso G^*$.
\end{proof}

We can now parameterize the isomorphism classes of finite groups $G$
that admit free actions on spheres.  First one specifies a
type~\ref{type-I}--\ref{type-VI}.  Then one specifies a positive
integer $a$, a subgroup $\Gbar$ of $(\Z/a)^*$, and possibly a subgroup
$\Gbar_0$ of $\Gbar$, all satisfying some constraints.  Then one
chooses one or more auxiliary parameters, constrained in terms of
properties of $\Gbar$.  Together with $a$, these specify~$g$, hence
the isomorphism type of~$G$.  The following theorem is proven by
combining theorem~\ref{thm-isomorphism-detection} with an analysis of what possibilities can
actually arise.  We will
write the parameters $a$, $\Gbar$ and $g$ in the order one chooses
them, rather than in the order used in theorem~\ref{thm-isomorphism-detection}.

The constraints on the choices of parameters are difficult to express
uniformly.  But the constraint on one auxiliary parameter,
called $b$, is uniform.  For each type we obtain a positive integer
$\bbar$ from the structure of $\Gbar$, and $b$ must be the product of
$\bbar$ and nontrivial powers of all the primes dividing $\bbar$.  We
will express this by saying ``$b$ is as above.''  

\begin{theorem}[Irredundant enumeration]
  \label{thm-irredundant-list}
  Suppose $G$ is a finite group admitting a free and isometric action
  on a sphere.  Then there is exactly one tuple
  $(\hbox{\ref{type-I}---\ref{type-VI}},a,\Gbar,g)$ or
  $(\ref{type-II},a,\Gbar,\Gbar_0,g)$ listed below, whose
  corresponding group (defined at the end) is isomorphic to~$G$.\qed
\end{theorem}

\noindent
$(\hbox{Type~\ref{type-I}},a,\Gbar,g=a b t)$ where
\begin{enumerate}
\item
  $a$ is odd.
\item
  $\Gbar$ is a cyclic subgroup of $(\Z/a)^*$ of order prime to $a$.
  Define $\bbar$ and $\tbar$ as the odd and $2$-power parts of $|\Gbar|$.
\item
  $b$ is as above and $t$ is a power of~$2$, larger than~$\tbar$
  if $\tbar\neq1$.
\end{enumerate}

\noindent
$(\hbox{Type~\ref{type-II}},a,\Gbar,g=a b t)$ or
$(\hbox{Type~\ref{type-II}},a,\Gbar,\Gbar_0,g=a b t)$ where
\begin{enumerate}
\item
  $a$ is odd.
\item
  $\Gbar$ is a subgroup of $(\Z/a)^*$ which is the direct product of
  a cyclic group of order prime to $2a$ and an elementary abelian $2$-group of
  rank${}\leq2$. 
  Define $\bbar$ and $\tbar$ as the orders of these factors.
\item
  $\Gbar_0$, if specified, is a subgroup of $\Gbar$ of index~$2$ (if
  $\tbar=4$) or index${}\leq2$ (otherwise).
\item
  $b$ is as above; and 
  $t=8$, unless $\Gbar_0$ was specified, in which case~$t$ is a
   power of $2$, larger than~$8$.
\end{enumerate}

\noindent
$(\hbox{Type~\ref{type-III}},a,\Gbar,g=8ab\theta)$ where
\begin{enumerate}
  \item
    $a$ is prime to~$6$.
  \item
    $\Gbar$ is a subgroup of $(\Z/a)^*$ which is the direct product
    of a cyclic $3$-group and a cyclic group of order prime to~$6a$.
    Define $\thetabar$ and~$\bbar$ as the orders of these factors.
  \item
    $b$ is as above and $\theta$ is a power of~$3$,
    larger than $\thetabar$.
\end{enumerate}

\noindent
$(\hbox{Type~\ref{type-IV}},a,\Gbar,g=16ab\theta)$ where
\begin{enumerate}
\item
 $a$ is prime to~$6$.
  \item
    $\Gbar$ is a subgroup of $(\Z/a)^*$ which is the direct
    product of a cyclic group of order prime to~$6a$ and a group of
    order $1$ or~$2$.
    Define~$\bbar$ as the order of the first factor.
  \item
    $b$ is as above and $\theta$ is a nontrivial
    power of~$3$.
\end{enumerate}

\noindent
    $(\hbox{Type~\ref{type-V}},a,\Gbar,g=120ab)$ where
\begin{enumerate}
  \item
    $a$ is prime to~$30$.
  \item
    $\Gbar$ is a cyclic subgroup of $(\Z/a)^*$ of order prime to~$30a$.
    Define~$\bbar$ as its order.
  \item
    $b$ is as above.
\end{enumerate}

\noindent
$(\hbox{Type~\ref{type-VI}},a,\Gbar,g=240ab)$ where
\begin{enumerate}
  \item
    $a$ is prime to~$30$.
  \item
    $\Gbar$ is a subgroup of $(\Z/a)^*$ which is the direct product
    of a cyclic group of order prime to $30a$ and a group of order $1$ or~$2$.
    Define~$\bbar$ as the order of the first factor.
  \item
    $b$ is as above.
\end{enumerate}

\smallskip
Here are instructions for building $G$.  In a sense this is a
constructive version of the proof of theorem~\ref{thm-isomorphism-detection}.  For all
types, start by taking cyclic groups $A$, $B$ with orders $a$, $b$.
Up to isomorphism of the domain, $B$ has a unique surjection to the
subgroup of $\Gbar\sset\Aut A$ of order $\bbar$.  Form the
corresponding semidirect product $A:B$.  Now we consider the six
cases.

The easiest is type~\ref{type-V}---just set
$G=2A_5\times(A:B)$.

For type~\ref{type-I}, we take a cyclic group $T$ of order $t$.
Just as for $B$, there is an essentially unique surjection from $T$ to
the subgroup of $\Gbar$ of order~$\tbar$.  Then $G$ is the semidirect
product $A:(B\times T)$.

For type~\ref{type-III} one takes a cyclic group $\Theta$ of order
$\theta$.  Just as for $B$, there is an essentially unique surjection
from $\Theta$ to the subgroup of $\Gbar$ of
order~$\thetabar$.  We also take $\Theta$ to act nontrivially on
$Q_8$.  (Up to conjugacy in $\Aut Q_8$ there is a unique nontrivial
action.) Then $G$ is the semidirect product $(Q_8\times
A):(B\times\Theta)$.  $B$ acts trivially on $Q_8$; this is forced
since $|B|$ is prime to $6$.

A type \ref{type-IV} or~\ref{type-VI} group is  got from
a type \ref{type-III} or~\ref{type-V} group by adjoining a suitable $4$-element $\phi$.  In both
cases, $\phi$ squares to the central involution, centralizes $B$, and 
acts on $A$ by the nontrivial
involution in $\Gbar$ (if one exists) or trivially (otherwise).
For type~\ref{type-VI}, $\phi$ 
acts on $2A_5$ by an outer automorphism of order~$2$, which is unique
up to $\Aut2A_5$.
For type~\ref{type-IV},
$\phi$ inverts $\Theta$ and acts on $Q_8$ by an involution that
inverts the action of $\Theta$.  Such an automorphism is unique up to
an automorphism of $Q_8$ that respects the $\Theta$-action.

One can describe type~\ref{type-II} groups in terms of
type~\ref{type-I} in a similar way, but it is easier to  build them
directly.  Take $T$ to be a quaternion group of order~$t$.  First
suppose $t=8$.  Then we did not specify $\Gbar_0$.  Up to automorphism
of the domain there is a unique surjection from $T$ to the subgroup
$\Tbar$ of $\Gbar$ of order~$\tbar$.  We take $G=A:(B\times T)$.  On
the other hand, 
suppose $t>8$, in which case we did specify $\Gbar_0$.  We write $T_0$ for the
index~$2$ cyclic subgroup of $T$.  Up to automorphism of the domain,
there is a unique surjection from $T$ to $\Tbar$ which carries
$T_0$ onto the $2$-part of $\Gbar_0$ (which has order $1$ or~$2$).
And again $G=A:(B\times T)$.

\end{document}